\documentclass{article}
\usepackage[utf8]{inputenc}
\usepackage{amssymb,amsmath,amsfonts}
\usepackage[all]{xy}
\usepackage{graphicx}
\usepackage{amsmath}
\usepackage{amsthm}
\usepackage{amsfonts}
\usepackage{amssymb}%
\usepackage{xcolor}
\providecommand{\U}[1]{\protect\rule{.1in}{.1in}}
\newtheorem{theorem}{Theorem}

\newtheorem{defi}[theorem]{Definition}
\newtheorem{example}[theorem]{Example}

\newtheorem{idea memo}[theorem]{Idea Memo}

\newtheorem{notation}[theorem]{Notation}

\newtheorem{remark}[theorem]{Remark}

\mathchardef\mhyphen="2D

\title{Category Algebras and States on Categories}
\author{Hayato Saigo\footnote{h\_saigoh@nagahama-i-bio.ac.jp}\\
Nagahama Institute of Bio-Science and Technology}

\begin{document}

\maketitle

%\section*{memo}

%\begin{itemize}
    %\item R[C]上の状態はRvalued stateだが、とくに、RがC上の代数で、状態がじっさいにはR全体でなくCvaluedなのが「普通」に扱われている状況。これは書かないけど認識（別の論文では各）
    %\item Eの bimodule structureちゃんと書く。
    %\item adjointable map 利用
    %\item たぶん、involution structureは圏とrigがcov/cov とcon/con　のときのみちゃんとなる。
%\end{itemize}

\begin{abstract}
The purpose of this paper is to build a new bridge between category theory and a generalized probability theory known as noncommutative probability or quantum probability, which was originated as a mathematical framework for quantum theory, in terms of states as linear functionals defined on category algebras. We clarify that category algebras can be considered as generalized matrix algebras and that states on categories as linear functionals defined on category algebras turn out to be generalized of probability measures on sets as discrete categories. Moreover, by establishing a generalization of famous GNS (Gelfand-Naimark-Segal) construction, we obtain representations of category algebras of $^{\dagger}$-categories on certain generalized Hilbert spaces which we call semi-Hilbert modules over rigs.
\end{abstract}

\section{Introduction}

In the present paper we study category algebras and states defined on arbitrary small categories to build a new bridge between category theory (see \cite{EM,MAC,AWO,LEI2} and references therein, for example) and noncommutative probability or quantum probability (see \cite{ACC,HO,BGV} and references therein, for example), a generalized probability theory which was originated as a mathematical framework for quantum theory.

A category algebra is, in short, a convolution algebra of functions on a category. For example, on certain categories called 
%(pre-)M\"{o}bius category \cite{LER, HAI, LM} or more generally 
finely finite category\cite{LEI}, which is a categorical generalization of locally finite poset, the convolution operation can be defined on the set of arbitrary functions and it becomes a unital algebra called incidence algebra. Many authors have studied the notions of M\"{o}bius inversion, which has been one of fundamental part of combinatorics since the pioneering work by Rota  \cite{ROT} on posets, in the context of incidence algebras on categories (\cite{CF,LER,CLL,HAI,LM,LEI}, for example).

There is another approach to obtain the notion of category algebra. As is well known, a group algebra is defined as a convolution algebra consisting of finite linear combinations of elements.
By generalization with replacing "elements" by "arrows", one can obtain another notion of category algebra (see \cite{HAI}, for example), which also includes monoid algebra (in particular polynomial algebras) and groupoid algebras as examples. Note that for a category with infinite number of objects, the algebra is not unital.

The category algebras we focus on in the present paper are unital algebras defined on arbitrary small categories, which are slightly generalized versions of algebras studied under the name of the ring of an additive category \cite{MIT}.
These category algebras include the ones studied in \cite{HAI} as subalgebras in general, and they coincide for categories with finite number of objects. Moreover, one of the algebras we study, called "backward finite category algebra", coincides with incidence algebras for combinatorically important cases originally studied in \cite{ROT}. 

The purpose of this paper is to provide a new framework for the  interplay between regions of mathematical sciences such as algebra, probability and physics, in terms of states as linear functional defined on category algebras. As is well known, quantum theory can be considered as a noncommutative generalization of probability theory. At the beginning of quantum theory, matrix algebras played a crucial role (see \cite{BOR} for example). In the present paper we clarify that category algebras can be considered as generalized matrix algebras and that the notions of states on categories as linear functionals defined on category algebras turns out to be a conceptual generalization of probability measures on sets as  discrete categories (For the case of states on groupoid algebras over the complex field $\mathbb{C}$ it is already studied \cite{CIM}). 

Moreover, by establishing a generalization of famous GNS (Gelfand-Naimark-Segal) construction \cite{GN,SEG} (as for the studies in category theoretic context, see \cite{JAC, PAR, YAU} for example), we obtain a  representation of category algebras of $^{\dagger}$-categories on certain generalized Hilbert spaces (semi-Hilbert modules over rigs), which can be considered as an extension of the result in \cite{CIM} for groupoid algebras over $\mathbb{C}$. This construction will provide a basis for the interplay between category theory, noncommutaive probability and other related regions such as operator algebras or quantum physics. 

\begin{notation}
In the present paper, categories are always supposed to be small\footnote{This assumption may be relaxed by applying some appropriate foundational framework.}. The set of all arrows in a category $\mathcal{C}$ is also denoted as $\mathcal{C}$. $|\mathcal{C}|$ denotes the set of all objects, which are identified with corresponding identity arrows, in $\mathcal{C}$. We also use the following notations:
\[
^{C'}{\mathcal{C}}_{C}:=\mathcal{C}(C,C'),\;
\mathcal{C}_{C}:=\sqcup_{C'\in |\mathcal{C}|}\mathcal{C}(C,C'),\;
^{C'}{\mathcal{C}}:=\sqcup_{C\in |\mathcal{C}|}\mathcal{C}(C,C'),
\]
where $\mathcal{C}(C,C')$ denotes the set of all arrows from $C$ to $C'$.
\end{notation}

\section{Category Algebras}

We introduce the notion of rig, module over rig, and algebra over rig in order to study category algebras in sufficient generality for various future applications in noncommutative probability, quantum physics and other regions of mathematical sciences such as tropical mathematics.

\begin{defi}[Rig]
A rig $R$ is a set with two binary operations called addition and multiplication such that

\begin{enumerate}
    \item $R$ is a commutative monoid with respect to  addition with the unit $0$,
    \item $R$ is a monoid with respect to multiplication with the unit $1$,
    \item $r''(r'+r)=r''r'+r''r, \: (r''+r')r=r''r+r'r$ holds for any $r,r',r''\in R$ (Distributive law),
    \item $0r=0, \: r0=0$ holds for any $r\in R$ % where $0$ denotes the additive unit
    (Absorption law)
     .
\end{enumerate}

\end{defi}

\begin{defi}[Module over Rig]
A commutative monoid $M$ under addition with unit $0$ together with a left action of $R$ on $M$ $(r,m)\mapsto rm$ is called a left module over $R$ if the action satisfies the following:
\begin{enumerate}
    \item $r(m'+m)=rm'+rm, \: (r'+r)m=r'm+rm$ for any $m,m'\in M$ and  $r,r'\in R$.
    \item $0m=0, \: r0=0 $ for any $m\in M$ and $r \in R$.
\end{enumerate}
Dually we can define the notion of right module over $R$.

Let $M$ is left and right module over $R$. $M$ is called $R$-bimodule if 
\[
r'(mr)=(r'm)r
\]
holds for any $r,r'\in R$ and $m\in M$.

The left/right action above is called the scalar multiplication.
\end{defi}

%Notice that by identifying $r$ and $1_{M}( )r$...
%For $R$ commutative....$rm=mr$

\begin{defi}[Algebra over Rig]
A bimodule $A$ over $R$ is is called an algebra over $R$ if it is also a rig with respect to its own multiplication which is compatible with scalar multiplication, i.e., 
\[
(r'a')(ar)=r'(a'a)r, \; (a'r)a=a'(ra)
\]
for any $a,a' \in A$ and $r,r'\in R$.
\end{defi}

Usually the term "algebra" is defined on rings and algebras are supposed to have negative elements. In this paper, we use the term algebra to mean the module over rig with multiplication. %The author believe that this abuse of term will not make confusion (and, it is quite natural to call these general entities as algebra). 

\begin{defi}[Category Algebra]
Let $\mathcal{C}$ be a category and $R$ be a rig. 
An $R$-valued function $\alpha$ defined on $\mathcal{C}$ is said to be of backward (resp. forward) finite propagation if for any object $C$ there are at most finite number of arrows in the support of $\alpha$ whose codomain (resp. domain) is $C$.
The module over $R$ consisting of all $R$-valued functions of backward (resp. forward) finite propagation together with the multiplication defined by
\[
({\alpha}' \alpha)(c'') = \sum_{\{(c',c)|\:c''=c'\circ c\}} {\alpha}'(c'){\alpha}(c), \:\: c,c',c''\in \mathcal{C}
\]
becomes an algebra over $R$ with unit $\epsilon$ defined by
\[
\epsilon (c)=  \begin{cases}
            1 & (c \in |\mathcal{C}|) \\
            0 & (otherwise)
            \end{cases},
\]
and is called the category algebra of backward (resp. forward) finite propagation $R_{0}[\mathcal{C}]$ (resp. $^{0}R[\mathcal{C}]$) of $\mathcal{C}$ over $R$. 
The algebra $^{0}R_{0}[\mathcal{C}]$ over $R$ defined as the intersection 
$R_{0}[\mathcal{C}] \cap{^{0}R[\mathcal{C}]}$ is called the category algebra of finite propagation of $\mathcal{C}$ over $R$. %A subalgebra or extension of $^{0}R_{0}[\mathcal{C}]$ is called as a category algebra of $\mathcal{C}$ over $R$.
\end{defi}

\begin{remark}
$^{0}R_{0}[\mathcal{C}]$ coincide with the algebra studied in \cite{MIT} if $R$ is a ring. %The author is grateful to Dr. Soichiro Fujii for pointing out from a category-theoretic viewpoint that it is appropiriate to consider $R_{0}[\mathcal{C}]$ and $^{0}R[\mathcal{C}]$. 
\end{remark}

In the present paper we focus on the category algebras $R_{0}[\mathcal{C}]$,$^{0}R[\mathcal{C}]$ and $^{0}R_{0}[\mathcal{C}]$ which are the same if $|\mathcal{C}|$ is finite, although other extentions or subalgebras of $^{0}R_{0}[\mathcal{C}]$ are also of interest (see Example \ref{groupoid algebra} and \ref{incidence algebra}).

\begin{notation}
In the following we use the term category algebra and the notation $R[\mathcal{C}]$ to denote either of category algebras $R_{0}[\mathcal{C}]$,$^{0}R[\mathcal{C}]$ and $^{0}R_{0}[\mathcal{C}]$.
\end{notation}

\begin{defi}[Indeterminates]
Let $R[\mathcal{C}]$ be a category algebra and $c\in \mathcal{C}$. The function $\iota^{c} \in \: R[\mathcal{C}]$ defined as
\[
\iota^{c}(c')=  \begin{cases}
            1 & (c'=c) \\
            0 & (otherwise)
            \end{cases}
\]
is called the indeterminate (See Example \ref{monoid algebra}) corresponding to $c$.
%(Note that $\iota^{c'\circ c}=\iota^{c'}\iota^{c}$, and that $\iota^{c'}\iota^{c}=0$ when $\rm{dom}(c')\neq\rm{cod}(c)$.)
\end{defi}

For indeterminates, it is easy to obtain the following:

\begin{theorem}[Caluculus of Indeterminates]
Let $c,c'\in \mathcal{C}$, $\iota^{c},\iota^{c'}$ be the corresponding indeterminates and $r\in R$. Then
\[
\iota^{c'}\iota^{c}=  \begin{cases}
            \iota^{c'\circ c} & (\rm{dom}(c')=\rm{cod}(c)) \\
            0 & (otherwise),
            \end{cases}
\]
\[
r\iota^{c}=\iota^{c}r.
\]

\end{theorem}

In short, a category algebra $R[\mathcal{C}]$ is an algebra of functions on $\mathcal{C}$ equipped with the multiplication which reflects the compositionality structure of $\mathcal{C}$. By the identification of $c\in \mathcal{C} \mapsto \iota^{c}\in \: R[\mathcal{C}]$, categories are included in category algebras.

%\begin{defi}
%(Representation of category)
%definition?cofuntor?
%\end{defi}

%\begin{theorem}
%(Represenstations of category and representations of category algebra are naturally one to one)
%\end{theorem}

Let us establish the basic notions for caluculation in category algebras:

\begin{defi}[Column, Row, Entry]
Let $\alpha \in R[\mathcal{C}]$ and $C,C' \in |\mathcal{C}|$. The elements $\alpha_{C},\; ^{C'}\alpha,\; ^{C'}\alpha_{C} \in R[\mathcal{C}]$ defined as

\[
\alpha_{C}(c) =\begin{cases}
                    \alpha(c) & (c\in\: \mathcal{C}_{C} ) \\
                    0 & (otherwise),
                 \end{cases}
\]
\[
^{C'}\alpha(c) =\begin{cases}
                    \alpha(c) & (c\in\: ^{C'}\mathcal{C} ) \\
                    0 & (otherwise),
                 \end{cases}
\]
\[
^{C'}\alpha_{C}(c) =\begin{cases}
                    \alpha(c) & (c\in\: ^{C'}\mathcal{C}_{C} ) \\
                    0 & (otherwise),
                 \end{cases}
\]
are called the the $C$-column, $C'$-row and $(C',C)$-entry of $\alpha$, respectively.
\end{defi}

Note that either of the data ${\alpha}_{C}$($C\in |\mathcal{C}|$) , $^{C'}{\alpha}$($C'\in |\mathcal{C}|$)  or $^{C'}{\alpha}_{C}$ ($C, C'\in |\mathcal{C}|$) determine $\alpha$.
Moreover, if $|\mathcal{C}|$ is finite,
\[
\alpha=\sum_{C,C'\in |\mathcal{C}|}\;
^{C'}\alpha_{C}.
\]

By definition, the following theorem holds:

\begin{theorem}[Polynomial Expression]\label{polynomial expression}
For any $\alpha \in R[\mathcal{C}]$
\[
^{C'}\alpha_{C}=\sum_{c \in ^{C'}\mathcal{C}_{C}}\alpha(c) \iota^c=\sum_{c \in ^{C'}\mathcal{C}_{C}}\iota^c\alpha(c).
\]

If $|\mathcal{C}|$ is finite,
\[
\alpha=\sum_{c \in \mathcal{C}} \alpha(c)\iota^c=\sum_{c \in \mathcal{C}}\iota^c\alpha(c).
\]

%When $\alpha \in R_{0}[\mathcal{C}]$,
%\[
%^{C'}\alpha=\sum_{c \in ^{C'}\mathcal{C}}\alpha(c)X^c
%=\sum_{C \in \mathcal{C}}\;^{C'}{\alpha_{C}}. 
%\]

%When $\alpha \in\;^{0}{R}[\mathcal{C}]$,
%\[
%\alpha_{C} =\sum_{c \in \mathcal{C}_{C}}\alpha(c) X^c
%=\sum_{C' \in \mathcal{C}}\;^{C'}{\alpha_{C}}. 
%\]
\end{theorem}

The formulae above %representation of entries
%\[
%^{C'}\alpha_{C}=\sum_{c \in ^{C'}\mathcal{C}_{C}}\alpha(c) X^c
%\]
clarify that category algebras are generalized polynomial algebra (see Example \ref{monoid algebra}).
On the other hand, the following theorem, which shows that category algebras are generalized matrix algebras (see Example \ref{incidence algebra}), also follows by definition:

\begin{theorem}[Matrix Calculus]\label{matrix calculus}
For any $\alpha, {\alpha}' \in R[\mathcal{C}]$, $C,C'\in |\mathcal{C}|$ and $r\in R$, the followings hold:
\[
(\alpha'+\alpha)_{C}=\alpha'_{C}+\alpha_{C}, \:^{C'}(\alpha'+\alpha)=^{C'}\alpha'+^{C'}\alpha,
\]
\[
^{C'}(\alpha'+\alpha)_{C}=^{C'}\alpha'_{C}+^{C'}\alpha_{C}
\]
\[
(r'\alpha r)_{C}=r'\;\alpha_{C}r,\:\:
^{C'}(r'\alpha r)=r'\;^{C'}\alpha r,\:\:
^{C'}(r'\alpha r)_{C}=r'\;^{C'}\alpha_{C}r
\]
\[
(\alpha'\alpha)_{C}=
\alpha'\; \alpha_{C}=
\sum_{C'' \in |\mathcal{C}|}\; \alpha'_{C''}\; ^{C''}\alpha_{C}
\]
\[
^{C'}(\alpha'\alpha)=
^{C'}\alpha'\; \alpha=
\sum_{C'' \in |\mathcal{C}|}\; ^{C'}\alpha'_{C''}\; ^{C''}\alpha
\]
\[
^{C'}(\alpha'\alpha)_{C}=
^{C'}\alpha'\; \alpha_{C}=
\sum_{C'' \in |\mathcal{C}|}\; ^{C'}\alpha'_{C''}\; ^{C''}\alpha_{C}.
\]
%To check the condition of equality...
\end{theorem}

The theorem above implies the following:

\begin{theorem}
$\alpha \in R[\mathcal{C}]$ is determined by its action on columns $\epsilon_{C}$/ rows  $^{C'}{\epsilon}$ of the unit $\epsilon$ for all $C,C'\in |\mathcal{C}|$.
\end{theorem}
\begin{proof}

Let $\alpha \in R[\mathcal{C}]$ and $\epsilon$ be the unit of $R[\mathcal{C}]$. Then by definition
\[
\alpha= \alpha \epsilon, \: \alpha =\epsilon \alpha
\]
holds and it implies
$\alpha_{C}= \alpha \epsilon_C , \: ^{C'}{\alpha}=\;^{C'}{\epsilon}\; \alpha,$ which determines $\alpha$.
\end{proof}

\begin{remark}
It is convenient to make use of a kind of "Einstein convention" in physics: Double  appearance of object indices which do not appear elsewhere means the sum over all objects in the category. For instance,
\[
^{C'}(\alpha' \alpha)_C=^{C'}\alpha'_{C''}\;^{C''}\alpha_{C}\]
means
\[
^{C'}(\alpha'\alpha)_{C}=
\sum_{C'' \in |\mathcal{C}|}\; ^{C'}\alpha'_{C''}\; ^{C''}\alpha_{C}.
\]
The notation is quite useful especially for category algebra $R[\mathcal{C}]$ where $|\mathcal{C}|$ is finite. In that case it is easy to show the decomposition of unit:
\[
\epsilon=\epsilon_{C}\;^{C}{\epsilon}.
\]
As a collorary, 
\[
\alpha' \alpha= \alpha' \epsilon \alpha =\alpha' \epsilon_{C}\;^{C}\epsilon \alpha =\alpha'_{C}\;^{C}\alpha,
\]
holds, which means that the multiplication can be interpreted as inner product of columns and rows. Hence, you can insert $_{C}\; ^{C}$ in formulae when $C$ does not appear elsewhere.
\end{remark}

\section{Example of Category Algebras}

Let us see some important examples of category algebras.

\begin{example}[Function Algebra]
Let $\mathcal{C}$ be a set as discrete category, i.e., a category whose arrows are all identities. Then $R[\mathcal{C}]$ is nothing but the $R$-valued function algebra 
on $|\mathcal{C}|$, where the operations are defined pointwise. 
\end{example}

When the rig $R$ is commutative such as $R=\mathbb{C}$, the function algebra is also commutative. On the other hand, a category algebra is in general noncommutative even if the rig is commutative. In this sense, category algebras can be considered as generalized (noncommutative) function algebras.

As we have noted, category algebras can also be considered as generalized polynomial algebras: 

\begin{example}[Monoid Algebra]\label{monoid algebra}
Let $\mathcal{C}$ be a monoid, i.e., a category with only one object. Then $R[\mathcal{C}]$ is the monoid algebra of $\mathcal{C}$. For example, in the case of  $\mathcal{C}=\mathbb{N}$ as  additive monoid, $R[\mathcal{C}]$ is the polynomial algebra over $R$.% For multivariate/noncommutative cases, consider the free commutative monoids/free monoids.
\end{example}

Since a monoid $\mathcal{C}$ has only one object, any $\alpha\in R[\mathcal{C}]$ can be presented as,  
\[
\alpha=\sum_{c \in \mathcal{C}} \alpha(c) \iota^c
\]
by Theorem \ref{polynomial expression}
which make it clear that $R[\mathcal{C}]$ is a generalized polynomial algebra.

As special cases of Example \ref{monoid algebra}, we have group algebras.

\begin{example}[Group Algebra]\label{group algebra}
Let $\mathcal{C}$ be a group, i.e., a monoid whose arrows are all invertible. Then $R[\mathcal{C}]$ coincides with the group algebra of $\mathcal{C}$. For example, in the case of  $\mathcal{C}=\mathbb{Z}$, $R[\mathcal{C}]$ is the Laurent polynomial algebra over $R$.% For multivariate/noncommutative cases, consider the free commutative groups/free groups.
\end{example}

By another generalization of Example \ref{group algebra} other than Example \ref{monoid algebra}, we have groupoid algebras.

\begin{example}[Groupoid Algebra]\label{groupoid algebra}
Let $\mathcal{C}$ be a groupoid, i.e., a category whose arrows are all invertible. When $|\mathcal{C}|$ is finite,   $R[\mathcal{C}]$ is nothing but the groupoid algebra of $|\mathcal{C}|$. Otherwise %When $|\mathcal{C}|$ is infinite, 
$R[\mathcal{C}]$ is a unital extension of the groupoid algebra in conventional sense which is nonunital. $R[\mathcal{C}]$ is quite useful to treat certain algebras which appeared in quantum physics \cite{CIM}. %IFS, QW...
(See Example \ref{adjacency algebra} also.)
\end{example}

As special cases of the the Example \ref{groupoid algebra} we have matrix algebras:

\begin{example}[Matrix Algebra]\label{adjacency algebra}Let $\mathcal{C}$ be an indiscrete category, i.e., a category such that for every pair of objects $C,C'$ there is exactly one arrow from $C$ to $C'$. Denote the cardinal of $|\mathcal{C}|$ is $n$. Then $R[\mathcal{C}]$ is isomorphic to the matrix algebra $M_n(R)$. %For the general simple undirected $\Gamma$, $R[\mathcal{C}]$ can be considerd as algebras consiting of certain type of matrices of finite propagation, i.e. matices whose rows and columns have finite supports. This kind of category algebras includes important algebras appear in quantum physics, e.g., the algebra of quantum harmonic oscillator, as subalgebras. 
\end{example}

%The perspective that category algebras are generalized matirix algebras leads to the consideration of many kind of (infinite) matrix algebras... upper/lower triangular matrices or O'mealla algebra...

Example \ref{adjacency algebra} above shows that matrix algebras are category algebras. Conversely, any category algebra can be considered as generalized matrix algebra (see Theorem \ref{matrix calculus}). This point of view is also useful to study quivers \cite{GAB}, i.e., directed graphs with multiple edges and loops.

\begin{example}[Path Algebra]\label{path algebra}
Let $\mathcal{C}$ be the free category of a quiver $Q$. $R[\mathcal{C}]$ coincides with the notion of path algebra when the quiver $Q$ has finite number of vertices. Otherwise the former includes the latter as a subalgebra. 
%The element of category algebra corresponding to edge of the digraph: Generalized Markov transition matices.
%In the Cayley graph of monoid....
%Generalization: Smooth paths
%Calculation by Matrices: Entries are "polynomials of edges". "Generating Matrix" for quiver.
\end{example}
%The special example above is related to Quantum Walks:
%\begin{example}
%"Time evolution operator matrix" $U$ of quantum walk: 
%\end{example}
Another important origin of the notion of category algebra is that of incidence algebra  (\cite{CF,LER,CLL,HAI,LM,LEI}, for example) originally studied on posets \cite{ROT}.

\begin{example}[Incidence Algebra] \label{incidence algebra}
Let $\mathcal{C}$ be a finely finite category \cite{LEI}, i.e., a category  such that for any $c \in \mathcal{C}$ there exist finite number of pairs of arrows $c',c'' \in \mathcal{C}$ satisfiyng $c=c'\circ c''$. Then $R^{\mathcal{C}}$, the set of all functions from $\mathcal{C}$ to $R$, becomes a unital algebra and called the incidence algebra of $\mathcal{C}$ over $R$.
%Question:Is there any good notion of "finite decomposition algebra"?
\end{example}

Let $\mathcal{C}$ be a category such that for any $C\in \mathcal{C}$ there exist at most finitely many arrows whose codomain is $C$. Then $R_0[\mathcal{C}]$ coincides with %$R^{\mathcal{C}}$, the set of all functions from $\mathcal{C}$.
%Especially, $R_0[\mathcal{C}]$ is  nothing but 
the incidence algebra on $\mathcal{C}$. % when $R$ is a unital commutative ring and $\mathcal{C}$ is the partially ordered set such that for any object $A$ there exist finite number of objects which is smaller than $A$ 
(One of the most classical examples is the poset consisting of all positive integers ordered by divisibility).
For the category satisfying the condition above, $R[\mathcal{C}]$ includes the zeta function $\zeta$ defined as
\[
\zeta(c)=1
\]
for all $c$. The multiplicative inverse of $\zeta$ is denoted as $\mu$ and called M\"{o}bius function. The relation  $\mu\zeta=\zeta\mu=\epsilon$ is a generalization of the famous M\"{o}bius inversion formula, which has been considered as the foundation of combinatorial theory since one of the most important papers in modern combinatorics \cite{ROT}. 

%At least conceputualy, this category algebra is related to operational calculus by Mikusinski, which is an algebraic theory of distributuion based on Titchmarch's theorem cite{MIK}, 
%(the analogous result is also known in whitenoise functional. cite{HOS}. )
%Another posibble direction is the study of Nest algebra used in operator algebra. cite{***}
%We will study the category algebraic perspective to such themes in analysis in another paper.

\section{States on Categories}

We will introduce the notion of states on categories to provide a foundation for stochastic theories on categories. As we will see, we can construct noncommutative probability space, a generalized notion of measure theoretic probability space based on category algebras. The key insight is that what we need to establish statistical law is the expectation functional, which is the functional which maps each random variable (or "observable" in the quantum physical context) to its expectation value. Considering a functional on $R[C]$ as expectation functional, we can interpret $R[C]$ as an algebra of noncommutative random variables, such as observables of quanta.   

\begin{defi}[Linear Functional]
Let $A$ be an algebra over a rig $R$. An $R$-valued linear function on $A$, i.e., a function preserving addtion and scalar multiplication, is called a linear functional on $A$.
A linear functional on $A$ is said to be unital if  $\varphi(\epsilon)=1$ where $\epsilon$ and $1$ denote the multiplicative unit in $A$ and $R$, respectively.
\end{defi}

\begin{defi}[Linear Functional on Category]
Let $R$ be a rig and $\mathcal{C}$ be a category. A  (unital) linear functional on $R[\mathcal{C}]$ is said to be an $R$-valued (unital) linear functional on the category $\mathcal{C}$.
\end{defi}

Although the main theme here is stochastic theory making use of positivity structure defined later, linear functionals on category algebras are used not only in the context with positivity. A very interesting example is "umbral calculus" \cite{RR}, an interesting tool in combinatorics, which can be interpreted as the theory of linear functionals on certain monoid algebras. Hence, studying the linear functionals on a category will lead to a generalization of umbral calculus. 

Given a linear functional on a category, we obtain a function on the set of arrows. For categories with a finite number of objects, we can characterize the former in terms of the latter:%the converse is also true:

\begin{theorem}[Linear Fuctional and Function]

Let $\varphi$ be a $R$-valued linear fuctional on $\mathcal{C}$. Then the function $\hat{\varphi}$ defined as
\[\hat{\varphi}(c)=\varphi(\iota^{c})
\]
becomes a $Z(R)$-valued function on $\mathcal{C}$, i.e., an R-valued function satisfying  $r\hat{\varphi}(c)=\hat{\varphi}(c)r$ for any $c\in \mathcal{C}$ and $r\in R$. Conversely, when $|\mathcal{C}|$ is finite, any $Z(R)$-valued function $\phi$ on $\mathcal{C}$ gives $R$-valued linear functional $\check{\phi}$ defined as
\[
\check{\phi}(\alpha)%=\sum_{C,C'}\sum_{c \in ^{C'}\mathcal{C}_{C}}\alpha(c) \phi(c)
=\sum_{c \in \mathcal{C}}\alpha(c)\phi(c)
=\sum_{c \in \mathcal{C}}\phi(c)\alpha(c)
\]
and the correspondence %$\varphi$ and $\phi$ 
is bijective.
\end{theorem}
\begin{proof}

Let $\varphi$ be a $R$-valued linear functional. Since $r\iota^{c}=\iota^{c}r$ for any $r\in R$ and $c\in \mathcal{C}$, we have $r\varphi(\iota^{c})=\varphi(\iota^{c})r$ which means $r\hat{\varphi}(c)=\hat{\varphi}(c)r$. The converse direction and bijectivity directly follows from definitions and Theorem \ref{polynomial expression}. 
\end{proof}

%Example: "cost" as functional.
As a collorary we also have the following:

\begin{theorem}[Unital Linear Functional and Normalized Function]
Let $\mathcal{C}$ be a category such that $|\mathcal{C}|$ is finite. Then there is one to one correspondence between $R$-valued unital linear functionals $\varphi$ 
and normalized $Z(R)$-valued functions $\phi$ 
on $\mathcal{C}$, i.e., $Z(R)$-valued functions $\phi$ satisfying 
\[
\sum_{C\in |\mathcal{C}|}\phi(C)=1.
\]
(Note that we identify objects and identity arrows.)
\end{theorem}
%Example: probability measure on discrete category.
%positivity???---->next
%\begin{defi}[Multiplicativity]
%a linear functional is called multipicative if...
%\end{defi}
%multiplicative cost. Tropical, Complex
%\begin{defi}[Idempotency]
%let linear functional $\varphi$ be an element of $R[\mathcal{C}]$ and satisfy $\varphi \varphi =\varphi$
%Then $\varphi$ is said to be idempotent. 
%\end{defi}
%remark: from idempotents to category. Cauchy completeness. "Category of (idempotent) States?": Operator algebra?
%cost category:multiplicative and idempotent case.
%path integration/least action on free category on graph.(Tropical case: degeneracy resolved by non-canonical choice of order)

To define the notion of state as generalized probability measure which can be applied in noncommutative contexts such as stochastic theory on category algebras, we need the notions of involution and positivity structure. 

\begin{defi}[Involution on Category]\label{involution on category}
Let $\mathcal{C}$ be a category. A covariant/contravariant endofunctor $(\cdot)^{\dagger}$ on $\mathcal{C}$ is said to be a covariant/contravariant involution on $C$ when $(\cdot)^{\dagger} \circ (\cdot)^{\dagger}$ is equal to the identity functor on $\mathcal{C}$. A category with contravariant involution which is identity on objects is called a $^{\dagger}$-category.
\end{defi}

\begin{remark}
For the studies on involutive categories, which are categories with involution satisfying certain conditions, see \cite{JAC,YAU} for example.

\end{remark}

\begin{defi}[Involution on Rig]\label{involution on rig}
Let $R$ be a rig. An operation $(\cdot)^{\ast}$ on $R$ preserving addition and covariant/contravariant with respect to multiplication is said to be a covariant/contravariant involution on $R$ when $(\cdot)^{\ast} \circ (\cdot)^{\ast}$ is equal to the identity function on $R$. A rig with contravariant involution is called a $^{\ast}$-rig.
\end{defi}

\begin{defi}[Involution on Algebra]\label{involution of algebra}
Let $A$ be an algebra over a rig $R$ with an covariant (resp. contravariant) involution $\overline{(\cdot)}$ . A covariant (resp. contravariant) involution $(\cdot)^{\ast}$ on $A$ as a rig is said to be a covariant (resp. contravariant) involution on $A$ as an algebra over $R$ if it is compatible with scalar multiplication, that is, 
\[
(r'ar)^{\ast}=\overline{r'}a^{\ast}\overline{r} \:\:\: \text{(covariant case)},\:\:\: 
(r'ar)^{\ast}=\overline{r}a^{\ast}\overline{r'} \:\:\: \text{(contravariant case)}. 
\]
An algebra $A$ over a $^{\ast}$-rig $R$ with contravariant involution is called a $^{\ast}$-algebra over $R$.
\end{defi}

\begin{theorem}[Category Algebra as Algebra with Involution]
Let $\mathcal{C}$ be a category with a covariant (resp. contravariant) involution $(\cdot)^{\dagger}$  and $R$ be a rig with a covariant (resp. contravariant) involution $\overline{(\cdot)}$. Then the category algebra $^{0}R_{0}[\mathcal{C}]$ becomes an algebra with covariant involution (resp. ${}^{\ast}$-algebra) over $R$.
\begin{proof}
The operation $(\cdot)^{\ast}$ defined as 
$\alpha^{\ast}(c)=\overline{\alpha({c^{\dagger}})}$
becomes a covariant (resp. contravariant) involution on $^{0}R_{0}[\mathcal{C}]$.
For the contravariant case,
\[
    (\alpha \beta)^{\ast}(c) 
    = \overline{\alpha \beta(c^{\dagger})}
    =\overline{\sum_{c^{\dagger}=c'\circ c''}\alpha(c')\beta(c'')}
    =\sum_{c^{\dagger}=c'\circ c''}\overline{\alpha(c')\beta(c'')}
    =\sum_{c^{\dagger}=c'\circ c''}\overline{\beta(c'')}\; \overline{\alpha(c')}
\]
which is equal to $\sum_{c={c''}^{\dagger}\circ {c'}^{\dagger}}\overline{\beta(c'')}\;\overline{\alpha(c')}$. By changing the labels of arrows it can be rewritten as
\[
\sum_{c={c''}^{\dagger}\circ {c'}^{\dagger}}\overline{\beta(c'')}\;\overline{\alpha(c')}
=\sum_{c=c'\circ c''}\overline{\beta({c'}^{\dagger})}\;\overline{\alpha({c''}^{\dagger})}
=\sum_{c=c'\circ c''}\beta^{\ast}(c') \alpha^{\ast}(c'')
=\beta^{\ast}\alpha^{\ast}(c).
\]
The proof for the covariant case is similar and more straightforward.
\end{proof}
\end{theorem}

Every category/rig has a trivial involution (identity). Thus, any category algebra $^{0}R_{0}[\mathcal{C}]$ can be considered as algebra with involution. In physics, especially quantum theory, the ${}^{\ast}$-algebra $^{0}R_{0}[\mathcal{C}]$ where $\mathcal{C}$ is a groupoid as $^{\dagger}$-category with inversion as involution and $R=\mathbb{C}$ as $^{\ast}$-rig with complex coujugate as involution. (For the importance of groupoid algebra in physics, see \cite{CIM} and references therein, for example).

Based on the involutive structure we can define the positivity structure on algebras:

\begin{defi}[Positivity]\label{positivity}
A pair of rigs with involution $(R,R_{+})$ is called a positivity structure on $R$ if $R_{+}$ is a subrig such that $r,s \in R_{+}$ and $r+s=0$
implies $r=s=0$, and that $a^{\ast}a\in R_{+}$ for any $a\in R$. % $R$ is said to be a rig with a positivity structure, and elements in $R_{+}$ is said to be positive (with respect to the positivity structure). 

\end{defi}

%Any rig $R$ has trivial positivity structure with $R_{+}=\{0\}$. 
%This is a generalization of the notion of positive cone in an algebra. 
The most typical examples are $(\mathbb{C},\mathbb{R}_{\geq 0})$,
$(\mathbb{R},\mathbb{R}_{\geq 0})$,
and $(\mathbb{R}_{\geq 0},\mathbb{R}_{\geq 0})$. Another interesting example is the tropical algebraic one $(\mathbb{R}\cup \{\infty\},\mathbb{R}\cup \{\infty\})$ where $\mathbb{R}\cup \{\infty\}$ is considered as a rig with respect to $\min$ and $+$.

\begin{defi}[State]\label{state}
Let $R$ be a rig with involution %(${}^{\ast}$-rig) 
and $(R,R_{+})$ be a positivity stucture on $R$. A state $\varphi$ on an algebra $A$ with involution %(${}^{\ast}$-algebra $A$)
over $R$ with respect to $(R,R_{+})$ is a unital linear functional $\varphi :  A\longrightarrow R$ which satisfies $\varphi(a^{\ast}a) \in R_{+}$ and 
$\varphi(a^{\ast})=\overline{\varphi(a)}$ for any $a\in R$, where $(\cdot)^{\ast}$ and $\overline{(\cdot)}$ denotes the  involution on $A$ and $R$, respectively.

%\begin{itemize}
 %   \item linearity: $\varphi(a+b)=\varphi(a)+\varphi(b)$, 
  %  \item unitality: $\varphi(1)=1$, and
   % \item positivity: $\varphi(a^{\ast}a) \in R_{+}$ 
%\end{itemize}
%for any $a,b \in R$, where $()^{\ast}$ denotes the involution.
%multiplicative, idempotent.
\end{defi}

\begin{remark}
The last condition $\varphi(a^{\ast})=\overline{\varphi(a)}$ follows from other conditions if $R=\mathbb{C}$.
\end{remark}

\begin{defi}[Noncommutative Probability Space]
A pair $(A,\varphi)$ consisting of an algebra $A$ with involution over a rig $R$ with involution and an $R$-valued state $\varphi$ is called a noncommutative probability space.
\end{defi}

There are many studies on noncommutative probability spaces where the algebra $A$ is a $^{\ast}$-algebra over $\mathbb{C}$. As is well known, the notion of noncommutative probability space  essentially includes the one of probability spaces in conventional sense, which corresponds to the cases that algebras $A$ are commutative $^{\ast}$-algebras (with certain topological structure). On the other hand, when the algebras are noncommutative, noncommutative probability spaces provide many examples which cannot be reduced to conventional probability spaces, such as models for quantum systems.

\begin{defi}[State on Category]

Let $R$ be a rig with involution and $(R,R_{+})$ be a positivity stucture on $R$. A state on the category algebra $^{0}R_{0}[\mathcal{C}]$ over $R$ with respect to $(R,R_{+})$ is said to be a state on a category $\mathcal{C}$ with respect to $(R,R_{+})$.
\end{defi}

As category algebras are in general noncommutative, states on categories provide many concrete noncommutative probability spaces generalizing such simplest examples as interacting Fock spaces \cite{AB} which are generalized harmonic oscillators, where the categories are indiscrete categories corresponding to certain graphs. 

The notion of state can be characterized for the categories with finite number of objects as follows:

\begin{theorem}[State and Normalized Positive Semidefinite Function]
Let $\mathcal{C}$ be a category such that $|\mathcal{C}|$ is finite. Then there is one to one correspondence between states $\varphi$ with respect to $(R,R_{+})$ and normalized positive semidefinite $Z(R)$-valued functions $\phi$ with respect to $(R,R_{+})$, i.e., normalized functions such that
\[
\sum_{\{(c,c')%\in \mathcal{C}\times \mathcal{C}
| \rm{dom}((c')^{\dagger})=\rm{cod}(c)\}}\overline{\xi(c')}\phi((c')^{\dagger}\circ c)\xi(c)
\]
is in $R_{+}$ for any function $\xi$ on $\mathcal{C}$ with finite support and that $\phi(c^{\dagger})=\overline{\phi(c)}$, where $(\cdot)^{\ast}$ and $\overline{(\cdot)}$ denotes the  involution on $A$ and $R$, respectively.
\end{theorem}
\begin{proof}

Note that a function $\xi$ on $\mathcal{C}$ with finite support can be considered as an element in $^{0}R_{0}[\mathcal{C}]$ and vice versa when $|\mathcal{C}|$ is finite. Then the theorem follows from the identity 
\begin{equation*}
\begin{split}
  \xi^{\ast}\xi
  &=(\sum_{c'\in \mathcal{C}}\overline{\xi((c')^{\dagger})}\iota^{c'})(\sum_{c\in\mathcal{C}}\iota^{c}\xi(c))\\
  &=(\sum_{c'\in \mathcal{C}}\overline{\xi(c')}\iota^{(c')^{\dagger}})(\sum_{c\in\mathcal{C}}\iota^{c}\xi(c))\\
  &=\sum_{\{(c,c')
  %\in \mathcal{C}\times \mathcal{C}
  | \rm{dom}((c')^{\dagger})=\rm{cod}(c)\}}\overline{\xi(c')}\iota^{(c')^{\dagger}\circ c}\xi(c).
\end{split}  
\end{equation*}
and the condition corresponding to $\varphi(\xi^{\ast})=\overline{\varphi(\xi)}$.
\end{proof}

The theorem above is a generalization of the result stated in the section 2.2.2 in \cite{CIM} for groupoid algebras over $\mathbb{C}$.
%For the case of $(R,R_{+})=(\mathbb{C},\mathbb{R}_{\geq 0})$ and $\mathcal{C}$ be a finite discrete category, the notion of normalized positive semidefinite function coincides with conventional one. On the other hand, 
For the case of discrete category, the notion coincides with the notion of probability measure on objects (identity arrows). Hence, the notion of state on category can be considered as noncommutative generalization of probability measure which is associated to the transition from set as discrete category (0-category) to general category (1-category).

Given a state on a $^{\dagger}$-category, we can construct a kind of GNS(Gelfand-Naimark-Segal) representation \cite{GN,SEG} (as for generalized constructions, see \cite{STE, PAS, JAC, PAR, YAU} for example) in a semi-Hilbert module defined below, 
%(cite{K. B. Sinha. "Quantum Stochastic Process and Noncommutative Geometry" cambridge university press, 2007 https://doi.org/10.1017/CBO9780511618529.006}), 
a generalization of Hilbert space:

\begin{defi}[Semi-Hilbert Module over Rig]
Let $R$ be a rig with involution $\overline{(\cdot)}$. A right module $E$ over $R$ equipped with a positive semidefinite sesquilinear form, i.e., a function $\langle \cdot | \cdot \rangle: E\times E\longrightarrow R$ satisfying
\[
\langle v''|v'r'+vr\rangle=\langle v''| v'\rangle r' +\langle v''|v\rangle r
\]
%\[
%\langle r'v''+rv'| v\rangle=r'\langle v''|v\rangle+r\langle v'|v\rangle
%\]
\[
\langle v' | v \rangle=\overline{\langle v | v' \rangle}
\]
\[
\langle v|v \rangle \in R_{+}
\]
for any $v,v',v'' \in E$ and $r,r'\in R$
is called a semi-Hilbert module over $R$.
%semi Hilbert, pre Hilbert and Hilbert.
\end{defi}

When a semi-Hilbert module over $E$ is also a left module over $R$, the set $\mathrm{End}(E)$ consisting of module endomorphisms over $R$ on $E$ becomes an algebra over $R$: The bimodule structure is given by $(r'Tr)(v)=r'T(rv)$, where 
$T\in  \mathrm{End}(E)$ and $r,r'\in R$.

\begin{theorem}[Generalized GNS Representation]
Let $A$ be an $^{\ast}$-algebra over a rig $R$ with involution $(\cdot )^{\ast}$. For any state $\varphi$ on $A$ with respect to $(R,R_{+})$,  there exist a semi-Hilbert module $E^{\varphi}$ over $R$ which is also a left $R$ module equipped with a positive semidefinite sesquilibear form $\langle \cdot | \cdot \rangle^{\varphi}$, an element $e^{\varphi} \in E^{\varphi}$ such that $\langle e^{\varphi}| e^{\varphi} \rangle^{\varphi}=1$, and a homomorphism $\pi^{\varphi}: A \longrightarrow \mathrm{End} (E^{\varphi})$ between algebras over $R$ such that
\[
\varphi(\alpha)=\langle e^{\varphi}| \pi^{\varphi}(\alpha) e^{\varphi} \rangle^{\varphi}
\]
and
\[
\langle v'|\pi^{\varphi}(\alpha)v\rangle^{\varphi}=\langle \pi^{\varphi}(\alpha^{\ast})v'|v\rangle^{\varphi}
\]
hold for any $\alpha \in A$ and $v,v'\in E^{\varphi}$.

\begin{proof}
Let $E^{\varphi}$ be the algebra $A$ itself as a module over $R$ equipped with $\langle \cdot | \cdot \rangle^{\varphi}$
defined by $\langle \alpha' | \alpha \rangle^{\varphi}=\varphi((\alpha')^{\ast}\alpha)$. It is easy to show that $\langle \cdot | \cdot \rangle^{\varphi}$ is a positive semidefinite sesquilinear form and satisfies $\varphi(\alpha)=\langle e^{\varphi}| \pi^{\varphi}(\alpha) e^{\varphi} \rangle^{\varphi},
$ and $\langle v'|\pi^{\varphi}(\alpha)v\rangle^{\varphi}=\langle \pi^{\varphi}(\alpha^{\ast})v'|v\rangle^{\varphi}$
where $\pi^{\varphi}$ denotes the homomorphism $\pi^{\varphi}: A \longrightarrow \mathrm{End} (E^{\varphi})$ defined by  $\pi^{\varphi}(\alpha)=\alpha(\cdot)$, the left multiplication by $\alpha$, and $e^{\varphi}$ denotes the unit $\epsilon$ of $A$ as an element of $E^{\varphi}$.
\end{proof}
\end{theorem}

\begin{remark}
When the rig $R$ is actually a ring, we can construct $^{\ast}$-representation of $A$ as follows (This idea is due to Malte Gerhold): We call an endomorphism $T$ on a semi-Hilbert module $E$ adjointable if there is a (not necessarily unique) adjoint, i.e., an endomorphism $T^{\ast}$ with $\langle v'|Tv \rangle=\langle T^{\ast}v'|v\rangle$ for any $v,v' \in E$. When $E$ is also a left $R$ module, the set of adjointable endomorphisms $\mathrm{Adj}(E)$ becomes a subalgebra over $R$ of $\mathrm{End}(E)$. The set $\mathrm{Nul}(E)=\{T|\langle v'|Tv\rangle =0, \forall v,v'\in E\}$ becomes a two-sided ideal in $\mathrm{Adj}(E)$. When $R$ is a ring, the quotient of $\mathrm{Adj}(E)$ by $\mathrm{Nul}(E)$ becomes a $^{\ast}$-algebra and we can construct the $^{\ast}$-representation of $A$, since we can show that the two "adjoints" of an endomorphism coincide up to some element of $\mathrm{Nul}(E)$ by taking subtraction of endomophisms and can define the "taking adjoint" as involution operation in the quotient. In more general cases (especially for the rigs such that the cancellation law for addition does not hold), the GNS construction might not necessarily lead to a $^{\ast}$-representations by adjointable endomorphisms.

%Let R cancellabele R
%consider Adj, Nul, Adj/Nul

%Adj/NUl becomes star algebra and we can construct star-hom from  "Abstract" star algebra to "Concrete" star algebra, that is operator algebra modulo degeneracy. (This idea is due to Malte Gerhold)

\end{remark}

When $A$ is a $^{\ast}$-algebra over $\mathbb{C}$, we can prove Cauchy-Schwarz inequality for semi-Hilbert space. Then the set $N^{\varphi}=\{\alpha \in A|\langle \alpha | \alpha \rangle^{\varphi}=0 \}$ becomes a subspace of $A$. By taking the quotient $E^{\varphi}=A/N^{\varphi}$, which becomes a pre-Hilbert space, we obtain the following "GNS(Gelfand-Naimark-Segal)" representation of $A$.

\begin{theorem}[GNS Representation]
Let $A$ be a $^{\ast}$-algebra over $\mathbb{C}$. For any state $\varphi$ on $A$ with respect to $(\mathbb{C},\mathbb{R}_{\geq 0})$, there exist a pre-Hilbert space $E^{\varphi}$ over $\mathbb{C}$ equipped with an inner product $\langle \cdot | \cdot \rangle^{\varphi}$, an element $e^{\varphi} \in E^{\varphi}$ such that $\langle e^{\varphi}| e^{\varphi} \rangle^{\varphi}=1$, and a homomorphism $\pi^{\varphi}: A \longrightarrow \mathrm{End} (E^{\varphi})$ between %*****$^{\ast}$-*****
 algebras over $R$ such that
\[
\varphi(\alpha)=\langle e^{\varphi}| \pi^{\varphi}(\alpha) e^{\varphi} \rangle^{\varphi}
\]
and
\[
\langle v'|\pi^{\varphi}(\alpha)v\rangle^{\varphi}=\langle \pi^{\varphi}(\alpha^{\ast})v'|v\rangle^{\varphi}
\]
hold for any $\alpha \in A$ and $v,v'\in E^{\varphi}$.
%where $\mathrm{End} (H^{\varphi})$ denotes the algebra consisting of module endomorphisms over $R$ on $E^{\varphi}$. 
\end{theorem}

By taking completion we have usual Hilbert space formulation popular in the context of quantum mechanics.

\begin{remark}
If the state $\varphi$ is fixed as "standard" one, such as "vacuum", the Dirac bra/ket notation becomes valid if we interpret as follows:

$|\alpha\rangle =\pi^{\varphi}(\alpha)$, 
$\langle \alpha|=\varphi(\alpha^{\ast}(\cdot))$, %=\langle$ |a^{\ast}\rangle| \cdot\rangle$, 
$\langle \alpha'|\alpha\rangle=\varphi(a^{\ast} b)$, 
$|0 \rangle =|\epsilon \rangle$ (vacuum).
\end{remark}

As colloraries of theorems above, we have the following results, which are extensions of the Theorem 1 in \cite{CIM}. :

\begin{theorem}[Generalized GNS Representation of $^{\dagger}$-Category]
Let $\mathcal{C}$ be a $^{\dagger}$-category and $R$ be a $^{\ast}$-rig. For any $\varphi$ be a state on $\mathcal{C}$ with respect to $(R,R_{+})$, there exist a semi-Hilbert module $E^{\varphi}$ over $R$ which is also a left $R$ module equipped with a sesquilinear form $\langle \cdot | \cdot \rangle^{\varphi}$, an element $e^{\varphi} \in E^{\varphi}$ such that $\langle e^{\varphi}| e^{\varphi} \rangle^{\varphi}=1$, and a homomorphism $\pi^{\varphi}:\; ^{0}R_{0}[\mathcal{C}] \longrightarrow \mathrm{End} (E^{\varphi})$ between %*****$^{\ast}$-*****
 algebras over $R$ such that
\[
\varphi(\alpha)=\langle e^{\varphi}| \pi^{\varphi}(\alpha) e^{\varphi} \rangle^{\varphi}
\]
and
\[
\langle v'|\pi^{\varphi}(\alpha)v\rangle^{\varphi}=\langle \pi^{\varphi}(\alpha^{\ast})v'|v\rangle^{\varphi}
\]
hold for any $\alpha \in A$ and $v,v'\in E^{\varphi}$.
\end{theorem}

\begin{theorem}[GNS Representation of $^{\dagger}$-Category]
Let $\mathcal{C}$ be a $^{\dagger}$-category. For any $\varphi$ be a state on $\mathcal{C}$ with respect to $(\mathbb{C},\mathbb{R}_{\geq 0})$, there exist a pre-Hilbert space $E^{\varphi}$ over $\mathbb{C}$ equipped with an inner product $\langle \cdot | \cdot \rangle^{\varphi}$, an element $e^{\varphi} \in E^{\varphi}$ such that $\langle e^{\varphi}| e^{\varphi} \rangle^{\varphi}=1$, and a homomorphism $\pi^{\varphi}:\; ^{0}R_{0}[\mathcal{C}] \longrightarrow \mathrm{End} (E^{\varphi})$ between %*****$^{\ast}$-*****
 algebras over $\mathbb{C}$ such that
\[
\varphi(\alpha)=\langle e^{\varphi}| \pi^{\varphi}(\alpha) e^{\varphi} \rangle^{\varphi}
\]
and
\[
\langle v'|\pi^{\varphi}(\alpha)v\rangle^{\varphi}=\langle \pi^{\varphi}(\alpha^{\ast})v'|v\rangle^{\varphi}
\]
hold for any $\alpha \in A$ and $v,v'\in E^{\varphi}$.
\end{theorem}

%category theory and operator algebras: adjunction???

%Semigroup with zero

%Cauchy completion

%\begin{remark}
%Trivial involution (identity): Real Hilbert space
%Tropical? 
%\end{remark}

%dynamical systems: Saigo paper

%state can be seen as density matirx

%tropical mathematics from the viewpoint of category algebra

%R:noncommutative? Hilbert Module theory. R as category algebra.

\section*{Acknowledgements}

The author is grateful to Prof. Hiroshi Ando, Dr. Soichiro Fujii and Ms. Misa Saigo for fruitful discussions and comments. This work  was partially supported by Research Origin for Dressed
Photon, JSPS KAKENHI (grant number 19K03608 and 20H00001) and JST CREST (JPMJCR17N2).

\end{document}